\documentstyle[11pt,amssymb,amstex]{amsart}

\newtheorem{theorem}{Theorem}
\newtheorem{lemma}{Lemma}

\newtheorem*{proof}{Proof}

\title[A nonlocal fractional Helmholtz equation]{A nonlocal fractional Helmholtz equation}
\begin{document}

\author{Mokhtar Kirane, Batirkhan Kh. Turmetov, Berikbol T. Torebek}

\address{Mokhtar Kirane\\
LaSIE, Facult\'{e} des Sciences, P\^{o}le Sciences et Technologies, Universit\'{e} de La Rochelle, Avenue M. Crepeau, 17042 La Rochelle Cedex, France}
\email{{\tt mkirane@@univ-lr.fr}}

\address {Batirkhan Kh. Turmetov \\
Akhmet Yasawi University,
161200, Sattarkhanov st., 29, Turkestan, Kazakhstan} \email{{\tt
batirkhan.turmetov@@ayu.edu.kz}}

\address{Berikbol T. Torebek\\
Institute of Mathematics and Mathematical Modeling, 050010,
Pushkin st., 125, Almaty, Kazakhstan} \email{{\tt
torebek@@math.kz}}

\maketitle
\begin{abstract}
In this paper we study some boundary value problems for a fractional
analogue of second order elliptic equation with an involution
perturbation in a rectangular domain. Theorems on existence and
uniqueness of a solution of the considered problems are proved by
spectral method.

\end{abstract}
\section{Introduction}

The paper is concerned with four boundary value problems
concerning the fractional analogue of Helmholtz equation with a
perturbation term of involution type in the space variable. We
obtain for them existence and uniqueness results based on the
Fourier method.

To describe the problems, let $\Omega  = \left\{ {(x,y) \in R^2: 0 < x < 1, {\rm{  - \pi}} < y < \pi} \right\}.$
We consider the equation
\begin{multline}\label{1.1}L_x^{2\alpha } u(x,y) + u_{yy} (x,y)
\\- \varepsilon u_{yy} \left( {x, -y} \right) - c^2 u(x,y) =
0,\left( {x,y} \right) \in \Omega ,\end{multline} where
$c,\varepsilon $ are real numbers, $L_x^{2\alpha }= D_x^\alpha D_x^\alpha,$
where $$D_x^\alpha  u\left( {x,y} \right) = \frac{1}{{\Gamma \left(
{1 - \alpha } \right)}}\int\limits_0^x {\left( {x - s} \right)^{ -
\alpha } \frac{{\partial u}}{{\partial s}}\left( {s,y} \right)}
ds$$ is the Caputo derivative of order $\alpha  \in
\left( {0,1} \right]$ of $u$ with respect to $x$ \cite{Kilbas2006}.

Regular solution of Equation \eqref{1.1} is a function $u \in
C\left( {\bar \Omega } \right),$ such that $D_x^\alpha
u,D_x^{2\alpha } u,u_{yy}  \in C\left( \Omega  \right).$

Since for $\alpha  = 1:$ $$L_x^2
+ \frac{{\partial ^2 }}{{\partial y^2 }} = \frac{{\partial ^2
}}{{\partial x^2 }} + \frac{{\partial ^2 }}{{\partial y^2 }} =
\Delta.$$
Therefore, Equation \eqref{1.1} is a nonlocal generalization of the Helmholtz equation, which at $\varepsilon  = 0$ coincides with the Helmholtz equation.\\
{\bf Problem D.} {\it Find in the domain $\Omega$ a regular
solution of Equation \eqref{1.1}, satisfying the following
boundary value conditions:}
\begin{equation}\label{1.2}u\left( {0,y} \right) = \varphi \left( y \right),u\left( {1,y} \right) = \psi \left( y \right),{\rm{ }} - \pi  \le y \le \pi ,\end{equation}
$$u\left( {x, - \pi } \right) = u\left( {x,\pi } \right) = 0,{\rm{ 0}} \le x \le 1.$$
{\bf Problem N.} {\it Find in the domain $\Omega$ a regular
solution of Equation \eqref{1.1}, such that $u_y \left( {x,y}
\right) \in C\left( {\bar \Omega } \right)$ and satisfying
conditions \eqref{1.2} and:}
$$u_y \left( {x, - \pi } \right) = u_y \left( {x,\pi } \right) = 0,{\rm{ 0}} \le x \le 1.$$\\
{\bf Problem P.} {\it Find in the domain $\Omega$ a regular
solution of Equation \eqref{1.1}, such that $u_y \left( {x,y}
\right) \in C\left( {\bar \Omega } \right)$ and satisfying
conditions \eqref{1.2} and:}
$$u\left( {x, - \pi } \right) = u\left( {x,\pi } \right){\rm{,}}u_y \left( {x, - \pi } \right) = u_y \left( {x,\pi } \right){\rm{, 0}} \le x \le 1.$$
{\bf Problem AP.} {\it Find in the domain $\Omega$ a regular
solution of Equation \eqref{1.1}, such that $u_y \left( {x,y}
\right) \in C\left( {\bar \Omega } \right)$ and satisfying
conditions \eqref{1.2} and:}
$$u\left( {x, - \pi } \right) =  - u\left( {x,\pi } \right){\rm{,}}u_y \left( {x, - \pi } \right) =  - u_y \left( {x,\pi } \right){\rm{, 0}} \le x \le 1.$$
Here $\varphi \left( y \right),\psi \left( y \right)$ are given
sufficiently smooth functions.

Before we describe our results, let us dwell a while on the
existing literature concerning equations with involution.
Differential equations with modified arguments are equations in
which the unknown function and its derivatives are evaluated with
modifications of the time or space variables; such equations are
called in general functional-differential equations. Among such
equation, one can single, equations with involution; to describe
them, let $\Gamma$ be an open or a closed curve in the complex
plane or the plane of real variables $x$ and $y.$

The homeomorphism $$a^2(t)=a(a(t))=t,\,\,t\in\Gamma,$$ is called a
Carleman shift (involution) \cite{Carleman}.

Various problems for equations with involution were investigated in \cite{Andreev:1987},
\cite{Khromov2011}.\\
Note that problems D, N and P for Equation \eqref{1.1} at
$\varepsilon  = 0$ were studied in \cite{TurmetovTorebek:2014},
\cite{TurmetovTO:2014}. Some questions of solvability of boundary
value problems with fractional analogues of the Laplace operator
were studied in \cite{Masaeva:2012}, \cite{Yakubovich:2012}.

\section{Solution of One-Dimensional Equation with Fractional Derivative}

Let $\mu$ be a positive real number, $S = \{ t:0 < t < 1\} ,$
$\bar S = \{ t:0 \le t \le 1\} .$ We consider the problem
\begin{equation}\label{2.1}D^{2\alpha } y\left( t
\right) - \mu ^2 y\left( t \right) = 0,\,t\in S,\end{equation}
\begin{equation}\label{2.2}y(0) = a,y(1) = b,\end{equation} where $a,b$ are real numbers.

A solution of problem \eqref{2.1} - \eqref{2.2} is the function
$y \in C(\bar S),$ such that  $D^\alpha
y \in C(\bar S),$ $D^{2\alpha } y \in C\left( S \right).$\\
\begin{lemma}\label{lem1} {\it The solution of problem \eqref{2.1} -
\eqref{2.2} exists, is unique and it can be written in the form
\begin{equation}\label{1*}y(t) = aC(\mu t) + bS(\mu t),\end{equation} where
\begin{equation}\label{2.3}C\left( {\mu t} \right) =
\frac{{E_{\alpha ,1} \left( \mu  \right)E_{\alpha ,1} \left( { -
\mu t^\alpha  } \right) - E_{\alpha ,1} \left( { - \mu }
\right)E_{\alpha ,1} \left( {\mu t^\alpha  } \right)}}{{2\mu
E_{2\alpha ,\alpha  + 1} \left( {\mu ^2 } \right)}},\end{equation}
\begin{equation}\label{2.4}S\left( {\mu t} \right) =
\frac{{t^\alpha  E_{2\alpha ,\alpha  + 1} \left( {\mu ^2
t^{2\alpha } } \right)}}{{E_{2\alpha ,\alpha  + 1} \left( {\mu ^2
} \right)}}.\end{equation} Here $$E_{\alpha ,\beta } \left( z
\right) = \sum\limits_{k = 0}^\infty {\frac{{z^k }}{{\Gamma \left(
{\alpha k + \beta } \right)}}}$$ is the Mittag - Leffler type
function \cite{Kilbas2006}.}\end{lemma}
\begin{proof} From \cite{TurmetovTorebek:2014} it is known that the general solution of equation \eqref{2.1} has the form \begin{equation}\label{2*}y\left( t \right) = D_1 E_{\alpha ,1} \left( { - \mu t^\alpha  } \right) + D_2 E_{\alpha ,1} \left( {\mu t^\alpha  } \right),\end{equation}
where $D_1 ,D_2 $ are arbitrary constants.

Substituting the function \eqref{2*} into the boundary conditions \eqref{2.2} for unknown coefficients $D_1$ and $D_2$ we get $$D_1  = \frac{{aE_{\alpha ,1} \left( \mu  \right) - aE_{\alpha ,1} \left( { - \mu } \right) + E_{\alpha ,1} \left( { - \mu } \right) - b}}{{E_{\alpha ,1} \left( \mu  \right) - E_{\alpha ,1} \left( { - \mu } \right)}},$$ $$D_2  = \frac{{b - aE_{\alpha ,1} \left( { - \mu } \right)}}{{E_{\alpha ,1} \left( \mu  \right) - E_{\alpha ,1} \left( { - \mu } \right)}}.$$
    Since $E_{\alpha ,1} \left( \mu  \right) - E_{\alpha ,1} \left( { - \mu } \right) = 2\mu E_{2\alpha ,\alpha  + 1} \left( {\mu ^2 } \right),$ after some transformations, the solution of the problem \eqref {2.1} - \eqref {2.2} can be reduced to the form \eqref {1*}. This proves the lemma.
\end{proof}
Furthermore, for any $0 < \alpha  < 1,$ consider the equation
\begin{equation}\label{3*}y''(t) - \mu D^{2 - \alpha } y(t) = 0, t\in S.\end{equation}
    The following statement is known (see \cite{Nakhushev}):
\begin{lemma}\label{lem1*} If the function $ y(t) \in C\left( \bar{S} \right)
\cap C^2 \left( S \right),y(t) \ne Const$ is a solution of
equation \eqref{3*}, then it can not attain its positive maximum
(negative minimum) within the segment $\bar{S}.$\end{lemma}

\begin{lemma}\label{lem2} \cite{Kilbas2006} For $E_{\alpha ,\beta }
\left( z \right)$ as $\left| z \right| \to \infty$ the following
asymptotic estimation holds \begin{equation}\label{2.5}E_{\alpha
,\beta } (z) = \frac{1}{\alpha }z^{\frac{{\left( {1 - \beta }
\right)}}{\alpha }} e^{z^{\frac{1}{\alpha }} }  - \sum\limits_{k =
1}^p {\frac{{z^{ - k} }} {{\Gamma \left( {\beta  - \alpha k}
\right)}}}  + O\left( {\frac{1}{{\left| z \right|^{p + 1} }} }
\right),\end{equation} where $\left| {\arg z} \right| \le \rho _1
\pi ,\rho _1  \in \left( {\frac{\alpha }{2},\min \left\{ {1,\alpha
} \right\}} \right),\alpha \in \left( {0,2} \right),$ and for
$\arg z = \pi$ \begin{equation}\label{2.6}E_{\alpha ,\beta } (z) =
\frac{1}{{1 + \left| z \right|}},\left| z \right| \to
\infty.\end{equation}\end{lemma} From Lemmas \ref{lem1*} and
\ref{lem2} follows
\begin{lemma}\label{lem3} For any $t \in [0,1]$ the following
inequalities hold: $$0 \le S\left( {\mu t} \right),C\left( {\mu t}
\right) \le 1.$$\end{lemma}

\section{Spectral properties of the perturbed Sturm-Liouville problem}

Application of the Fourier method for solving problems D, N, P, AP
leads to the spectral equation
\begin{equation}\label{3.1}Y''\left( y \right) -
\varepsilon Y''\left( { - y} \right) + \lambda Y\left( x \right) =
0, - \pi  < y < \pi ,\end{equation} supplemented with one of the
local \begin{equation}\label{3.2}Y\left( { - \pi }
\right) = Y\left( \pi  \right) = 0,\end{equation}
\begin{equation}\label{3.3}Y'\left( { - \pi } \right) = Y'\left(
\pi  \right) = 0,\end{equation} or nonlocal
\begin{equation}\label{3.4}Y\left( { - \pi } \right) = Y\left( \pi
\right),Y'\left( { - \pi } \right) = Y'\left( \pi
\right),\end{equation}
\begin{equation}\label{3.5}Y\left( { - \pi } \right) =
- Y\left( \pi  \right),Y'\left( { - \pi } \right) =  - Y'\left(
\pi  \right)\end{equation} boundary conditions.

The Sturm-Liouville problem for Equation \eqref{3.1} with one of
the boundary conditions \eqref{3.2}, \eqref{3.3}, \eqref{3.4},
\eqref{3.5} is self-adjoint so they have real eigenvalues, and
their corresponding eigenfunctions form a complete orthonormal
basis in $L_2 \left( { - \pi ,\pi } \right)$ \cite{Naimark}.

For further investigation of the problems under consideration, we
need to calculate the explicit form of the eigenvalues and
eigenfunctions.

For $\left| \varepsilon  \right| < 1$  problem \eqref{3.1},
\eqref{3.2} has the following eigenvalues:
$$\lambda _{2k - 1,1}^{}  = \left( {1 + \varepsilon } \right)k^2,\quad
\lambda _{2k,1}^{}  = \left( {1 - \varepsilon } \right)\left( {k -
\frac{1}{2}} \right)^2 ,k = 1,2,...,$$ with the corresponding
eigenfunctions \begin{equation}\label{3.6}Y_{2k - 1,1}^{}  = \sin
ky,\quad Y_{2k,1}^{}  = \cos \left( {k - \frac{1}{2}} \right)y,\,k
= 1,2,....\end{equation} Similarly, problem \eqref{3.1},
\eqref{3.3} has the eigenvalues
$$\lambda _{2k + 1,2}^{}  = \left( {1 + \varepsilon }
\right)\left( {k + \frac{1}{2}} \right)^2 , \quad\lambda
_{2k,2}^{}  = \left( {1 - \varepsilon } \right)k^2 ,k = 0,1,...,$$
with the corresponding eigenfunctions
\begin{equation}\label{3.7}Y_{2k + 1,2}^{}  = \sin \left( {k +
\frac{1}{2}} \right)y,Y_{2k,2}^{}  = \cos ky,\,k =
0,1,....\end{equation} The eigenvalues of problem \eqref{3.1},
\eqref{3.4} are $$\lambda _{2k - 1,3}^{}  = \left( {1 +
\varepsilon } \right)k^2 ,k = 1,2,...,\quad\lambda _{2k,3}^{}  =
\left( {1 - \varepsilon } \right)k^2 ,k = 0,1,...,$$ with the
corresponding eigenfunctions \begin{equation}\label{3.8}Y_{2k -
1,3}^{}  = \sin ky,k = 1,2,...,Y_{2k,3}^{}  = \cos ky,\,k =
0,1,....\end{equation} Problem \eqref{3.1}, \eqref{3.5} has the
following eigenvalues $$\lambda _{2k + 1,4}^{}  = \left( {1 +
\varepsilon } \right)\left( {k + \frac{1}{2}} \right)^2 ,
\quad\lambda _{2k,4}^{}  = \left( {1 - \varepsilon } \right)\left(
{k + \frac{1}{2}} \right)^2 ,k = 0,1,...,$$ and corresponding
eigenfunctions \begin{equation}\label{3.9}Y_{2k + 1,4}^{}  = \sin
\left( {k + \frac{1}{2}} \right)y,Y_{2k,4}^{}  = \cos \left( {k +
\frac{1}{2}} \right)y,\,k = 0,1,....\end{equation}
\begin{lemma}\label{lem4}
The systems of functions \eqref{3.6}, \eqref{3.7},
\eqref{3.8}, \eqref{3.9} are complete and orthonormal in $L_2
\left( { - \pi ,\pi } \right).$\end{lemma}
\begin{proof}We prove only the completeness of system \eqref{3.6} in $L_2(-\pi, \pi).$ We will prove
that from the equalities
$$\int\limits_{-\pi}^{\pi}f(y)sinky dy=0, k=1,2,…,$$ $$\int\limits_{-\pi}^{\pi}f(y)cos\left(k-\frac{1}{2}\right)y dy=0, k=1,2,…,$$ for $f\in L_2(-\pi, \pi),$ we should obtain $f(y)=0$ in $L_2(-\pi, \pi).$

Suppose that the second equation holds. We transform it as follows

$$0=\int\limits_{-\pi}^{\pi}f(y)cos\left(k-\frac{1}{2}\right)y dy=\int\limits_{0}^{\pi}(f(y)+f(-y))cos\left(k-\frac{1}{2}\right)y dy.$$
Then by the completeness of the system \cite{Moiseev} $\left\{cos\left(k-\frac{1}{2}\right)y\right\}_{k=1}^{\infty}$ in $L_2(-\pi, \pi)$ we have $f(y)=-f(-y),\,0<y<\pi.$

Similarly
$$0=\int\limits_{-\pi}^{\pi}f(y)sinky dy=\int\limits_{0}^{\pi}(f(y)-f(-y))sinky dy.$$
Then by the completeness of the system \cite{Moiseev} $\left\{sinky\right\}_{k=1}^{\infty}$ in $L_2(-\pi, \pi)$ we have $f(y)= f(-y),\,0<y<\pi.$ Then we obtain $f(y)=0$ in $L_2(0, \pi),$ and consequently $f(y)=0$ in $L_2(-\pi, \pi).$

The completeness of the systems \eqref{3.7},
\eqref{3.8} and \eqref{3.9} can be proved similarly.
\end{proof}

\section{Main results}

For the considered problems D, N, P, AP, the following theorems
hold.

Suppose $\varphi _{2k - j,i}  = \left( {\varphi ,Y_{2k - j,i} }
\right),\psi _{2k - j,i}  = \left( {\psi ,Y_{2k - j,i} } \right),\,j
= 0,1,\,i = 1,2,3,4$ are the Fourier coefficients of functions $\varphi
,\psi ,$ by system $Y_{2k - j,i},$ $\lambda _{2k - j,i}$ are the
corresponding eigenvalues, $\mu _{2k - j,i}^2  = \lambda _{2k -
j,i}  + c^2,$ and functions $C\left( {\mu _{2k - j,i} x} \right)
$ and $S\left( {\mu _{2k - j,i} x} \right)$ are defined by \eqref{2.3} and \eqref{2.4}.

\begin{theorem}\label{th1} Let $|\varepsilon | < 1,\,0 < \delta  < 1,$
$\varphi \left( y \right) \in C^{2 + \delta } \left[ { - \pi ,\pi
} \right],$ $\psi \left( y \right) \in C^{1 + \delta } \left[ { -
\pi ,\pi } \right]$ and $\varphi \left( { - \pi } \right) =
\varphi \left( \pi  \right) = 0,$ $\psi \left( { - \pi } \right) =
\psi \left( \pi  \right) = 0.$ Then the solution of the problem D
exists, is unique and it can be written in the form \begin{multline*}u\left(
{x,y} \right) = \sum\limits_{k = 1}^\infty  {\left[ {\varphi _{2k
- 1,1} C\left( {\mu _{\mu _{2k - 1,1} } x} \right) + \psi _{2k -
1,1} S\left( {\mu _{\mu _{2k - 1,1} } x} \right)} \right]} Y_{2k -
1,1} \left( y \right)\\+ \sum\limits_{k = 1}^\infty  {\left[
{\varphi _{2k,1} C\left( {\mu _{\mu _{2k,1} } x} \right) + \psi
_{2k,1} S\left( {\mu _{\mu _{2k,1} } x} \right)} \right]} Y_{2k,1}
\left( y \right).\end{multline*}\end{theorem}

\begin{theorem}\label{th2} Let $|\varepsilon | <
1,\,0 < \delta  < 1,$ $\varphi \left( y \right) \in C^{3 + \delta }
\left[ { - \pi ,\pi } \right],$ $\psi \left( y \right) \in C^{2 +
\delta } \left[ { - \pi ,\pi } \right]$ and $\varphi '\left( { -
\pi } \right) = \varphi '\left( \pi  \right) = 0,$ $\psi '\left( {
- \pi } \right) = \psi '\left( \pi  \right) = 0.$ Then the
solution of the problem N exists, is unique and it can be written
in the form \begin{multline*}u\left( {x,y} \right) = \left( {1 - x^\alpha  }
\right)\varphi _{0,2}  + x^\alpha  \psi _{0,2} \\+ \sum\limits_{k =
0}^\infty  {\left[ {\varphi _{2k - 1,2} C\left( {\mu _{2k - 1,2}
x} \right) + \psi _{2k - 1,2} S\left( {\mu _{2k - 1,2} x} \right)}
\right]} Y_{2k - 1,2}\\ +\sum\limits_{k = 1}^\infty
{\left[ {\varphi _{2k,2} C\left( {\mu _{2k,2} x} \right) + \psi
_{2k,2} S_{2k,2} \left( {\mu _{2k,2} x} \right)} \right]} Y_{2k,2}
.\end{multline*}\end{theorem}

\begin{theorem}\label{th3} Let $|\varepsilon | < 1,\,0 < \delta  <
1,$ $\varphi \left( y \right) \in C^{3 + \delta } \left[ { - \pi
,\pi } \right],$ $\psi \left( y \right) \in C^{2 + \delta }
\left[ { - \pi ,\pi } \right]$ and $\varphi \left( { - \pi }
\right) = \varphi \left( \pi  \right),\,\varphi '\left( { - \pi }
\right) = \varphi '\left( \pi  \right),$ $\psi \left( { - \pi }
\right) = \psi \left( \pi  \right),\,\psi '\left( { - \pi } \right)
= \psi '\left( \pi  \right).$ Then the solution of the problem P
exists, is unique and it can be written in the form \begin{multline*}u\left(
{x,y} \right) = \left( {1 - x^\alpha  } \right)\varphi _{0,3} + x^\alpha  \psi _{0,3}\\ +
\sum\limits_{k = 1}^\infty  {\left[ {\varphi _{2k,3} C\left( {\mu
_{2k,3} x} \right) + \psi _{2k,3} S\left( {\mu _{2k,3} x} \right)}
\right]} Y_{2k,3} \left( y \right) \\ + \sum\limits_{k = 1}^\infty  {\left[ {\varphi _{2k - 1,3} C\left(
{\mu _{2k - 1,3} x} \right) + \psi _{2k - 1,3} S\left( {\mu _{2k -
1,3} x} \right)} \right]} Y_{2k - 1,3} \left( y \right).\end{multline*}\end{theorem}

\begin{theorem}\label{th4} Let $|\varepsilon | < 1,\,0 < \delta  < 1,$ 
$\varphi \left( y \right) \in C^{3 + \delta } \left[ { - \pi ,\pi
} \right],$ $\psi \left( y \right) \in C^{2 + \delta } \left[ { -
\pi ,\pi } \right]$ and $\varphi \left( { - \pi } \right) =  -
\varphi \left( \pi  \right),\,\varphi '\left( { - \pi } \right) =  -
\varphi '\left( \pi  \right),$ $\psi \left( { - \pi } \right) =  -
\psi \left( \pi  \right),\,\psi '\left( { - \pi } \right) =  - \psi
'\left( \pi  \right).$ Then the solution of the problem AP exists,
is unique and it can be written in the form \begin{multline*}u\left( {x,y}
\right) = \sum\limits_{k = 0}^\infty  {\left[ {\varphi _{2k,4}
C\left( {\mu _{2k,4} x} \right) + \psi _{2k,4} S\left( {\mu
_{2k,4} x} \right)} \right]} Y_{2k,4} \left( y \right)\\ +
\sum\limits_{k = 0}^\infty  {\left[ {\varphi _{2k - 1,4} C\left(
{\mu _{2k - 1,4} x} \right) + \psi _{2k - 1,4} S\left( {\mu _{2k -
1,4} x} \right)} \right]} Y_{2k - 1,4} \left( y \right).\end{multline*}\end{theorem}

\section{Proofs of the main results}

As the proofs for the uniqueness of the solutions of each problems
are similar, we will present only the proof for problem D.

As the system of eigenfunctions \eqref{3.6} of problem D forms an orthonormal basis in $L_2( - \pi ,\pi ),$ the function can be represented as follows \begin{equation}\label{5.1} u\left( {x,y} \right) = \sum\limits_{k = 1}^\infty  {u_{2k - 1,1} \left( x \right)Y_{2k - 1,1} \left( y \right)}  + \sum\limits_{k = 1}^\infty  {u_{2k,1} \left( x \right)Y_{2k,1} \left( y \right)},\end{equation} where $u_{2k - 1,1} (x),u_{2k,1} (x)$ are unknown coefficients. It is well known that if $\varphi \left( y \right),\psi \left( y \right)$ satisfy the conditions of Theorem \ref{th1}, then they can be uniquely represented in the form of a uniformly and absolutely convergent Fourier series by the systems $\left\{ {Y_{2k - 1,1} (y),Y_{2k,1} (y)} \right\}:$ $$\varphi \left( y \right) = \sum\limits_{k = 1}^\infty  {\varphi _{2k - 1,1} Y_{2k - 1,1}^{} \left( y \right)}  + \sum\limits_{k = 1}^\infty  {\varphi _{2k,1} Y_{2k,1}^{} \left( y \right)},$$ $$\psi \left( y \right) = \sum\limits_{k = 1}^\infty  {\psi _{2k - 1,1} Y_{2k - 1,1}^{} \left( y \right)}  + \sum\limits_{k = 1}^\infty  {\psi _{2k,1} Y_{2k,1}^{} \left( y \right)},$$ where $\varphi _{2k - j,1}  = \left( {\varphi ,Y_{2k - j,1}^{} } \right),$ $\psi _{2k - j,1}  = \left( {\psi ,Y_{2k - j,1}^{} } \right),j = 0,1.$\\
Putting \eqref{5.1} into Equation \eqref{1.1} and boundary
conditions \eqref{1.2}, for finding unknown functions $u_k \left(
x \right)$, we obtain the following problem
\begin{equation}\label{5.2}D_y^{2\alpha } u_{2k - j,1} \left( x
\right) - \mu _{2k - j,1}^2 u_{2k - j,1} \left( x \right) = 0,0 <
x < 1,\end{equation} \begin{equation}\label{5.3}u_{2k - j,1}
\left( 0 \right) = \varphi _{2k - j,1},\quad u_{2k - j,1} \left( 1
\right) = \psi _{2k - j,1} ,\end{equation} where $\mu _{2k -
j,1}^2  = \lambda _{2k - j,1}^{}  + c^2,$ $j = 0,1.$

Due to Lemma \ref{lem1} the solution of problem \eqref{5.2}-\eqref{5.3}
exists, is unique and it can be written in the form $$u_{2k - j,1}
\left( x \right) = \varphi _{2k - j,1} C\left( {\mu _{2k - j,1} x}
\right) + \psi _{2k - j,1} S\left( {\mu _{2k - j,1} x} \right),$$
where $C\left( {\mu _{2k - j,1} x} \right)$ and $S\left( {\mu _{2k
- j,1} x} \right)$ are defined by \eqref{2.3} and \eqref{2.4},
respectively. Furthermore, according to Lemma \ref{lem3}  inequalities $$0
\le S\left( {\mu _{2k - j,1} x} \right),C\left( {\mu _{2k - j,1}
x} \right) \le 1,x \in \left[ {0,1} \right]$$ are true.

Further, if the function $f\left( x \right)$ belongs to the class
$C^{m + \delta } \left[ {a,b} \right],$ $m = 0,1,...,0 < \delta  <
1,$ then for Fourier coefficients of this function the following
estimation holds (see.\cite{IlinPoznyak:2002}): $$\left| {f_k }
\right| = O\left( {\frac{1}{{k^{m + \delta } }}} \right),k \to
\infty.$$ If $\varphi ''\left( y \right) \in C^\delta  \left[ { -
\pi ,\pi } \right],$ $\psi '\left( y \right) \in C^\delta  \left[
{ - \pi ,\pi } \right]$ and conditions $\varphi \left( { - \pi }
\right) = \varphi \left( \pi  \right) = \psi \left( { - \pi }
\right) = \psi \left( \pi  \right) = 0$ hold, then $$\left|
{\varphi _{2k - 1,1} } \right| \le \frac{C}{{k^{2 + \delta }
}},\left| {\varphi _{2k,1} } \right| \le \frac{C}{{\left( {k -
\frac{1}{2}} \right)^{2 + \delta } }},$$ $$\left| {\psi _{2k -
1,1} } \right| \le \frac{C}{{k^{1 + \delta } }},\left| {\psi
_{2k,1} } \right| \le \frac{C}{{\left( {k - \frac{1}{2}}
\right)^{1 + \delta } }},\quad C = const.$$ For such functions,
we obtain \begin{gather}\begin{gathered}\left| {u_{2k - 1,1} \left( x
\right)} \right| \le C\left( {\frac{1}{{k^{2 + \delta } }} +
\frac{1}{{k^{1 + \delta } }}} \right),\\ \left| {u_{2k - 1,1}
\left( x \right)} \right| \le C\left( {\frac{1}{{\left( {k -
\frac{1}{2}} \right)^{2 + \delta } }} + \frac{1}{{\left( {k -
\frac{1}{2}} \right)^{1 + \delta } }}} \right).\label{5.4}\end{gathered}\end{gather}
Then the series \eqref{5.1} converges uniformly in the domain
$\bar \Omega$ and therefore $u\left( {x,y} \right) \in C\left(
{\bar \Omega } \right).$   Further, using estimations \eqref{2.5}
and \eqref{2.6}, we get $$S_{2k - j,1} \left( {\mu _{2k - j,1} x}
\right) = O\left( {e^{\mu _{2k - j,1}^{\frac{1}{\alpha }} \left(
{x - 1} \right)} } \right),$$ $$C\left( {\mu _{2k - j,1} x}
\right) = O\left( {\frac{1}{{\mu _{2k - j,1} }}} \right).$$ Taking
derivative term by term from the series \eqref{5.1} twice by $y$,
we have
$$u_{yy} \left( {x,y} \right) =  - \sum\limits_{k = 1}^\infty  {\lambda _{2k - 1,1} u_{2k - 1,1} \left( x \right)Y_{2k - 1,1} \left( y \right)}  - \sum\limits_{k = 1}^\infty  {\lambda _{2k,1} u_{2k,1} \left( x \right)Y_{2k,1} \left( y \right)} .$$ Then for all $x \ge x_0  > 0,$ $0 \le y \le 1,$ taking into account inequalities \eqref{5.4}, we have \begin{multline*}\left| {u_{yy} \left( {x,y} \right)} \right| \le C\sum\limits_{k = 1}^\infty  {\left( {\left| {\lambda _{2k - 1,1}^{} } \right|\left| {u_{2k - 1,1} \left( x \right)} \right| + \left| {\lambda _{2k,1}^{} } \right|\left| {u_{2k,1} \left( x \right)} \right|} \right)}\\ \le C\sum\limits_{k = 1}^\infty  {\left( {\left| {\lambda _{2k - 1,1}^{} } \right|\left| {u_{2k - 1,1} \left( x \right)} \right| + \left| {\lambda _{2k,1}^{} } \right|\left| {u_{2k,1} \left( x \right)} \right|} \right)}\\ \le C\sum\limits_{k = 1}^\infty  {k^{ - 1 - \delta }  + k^{1 - \delta } e^{ - \mu _{2k - 1,1} \left( {1 - x} \right)}  + \left( {k - \frac{1}{2}} \right)^{ - 1 - \delta } }  \\+ C\sum\limits_{k = 1}^\infty  {\left( {k - \frac{1}{2}} \right)^{1 - \delta } e^{ - \mu _{2k,1} \left( {1 - x} \right)}  < \infty } .\end{multline*} Similarly, estimate the series $$D_x^{2\alpha } u\left( {x,y} \right) = \sum\limits_{k = 1}^\infty  {\left( {\mu _{2k - 1,1}^2 u_{2k - 1,1} \left( x \right)Y_{2k - 1,1} \left( y \right) + \mu _{2k,1}^2 u_{2k,1} \left( x \right)Y_{2k,1} \left( y \right)} \right)} .$$ Then $u_{yy} \left( {x,y} \right),D_x^{2\alpha } u\left( {x,y} \right) \in C\left( \Omega  \right).$\\
The uniqueness of the solution of problem D follows from the
uniqueness of the solution of problem \eqref{5.2} - \eqref{5.3}.
The theorem is proved.

\section{Acknowledgement}
The final version of this paper was completed when Berikbol T. Torebek was visiting the University of La Rochelle (France).
The research of Turmetov and Torebek is financially supported by a grant from the
Ministry of Science and Education of the Republic of Kazakhstan
(Grants No. 0819/GF4). The research of Kirane is supported by NAAM research group, University of King Abdulaziz, Jeddah.

\end{document}